\documentclass{amsart}

\usepackage{amsmath}
\usepackage{amsfonts}
\usepackage{amssymb,enumerate}
\usepackage{amsthm}
\usepackage[all]{xy}
\usepackage{hyperref}

\newcommand{\Kr}{\operatorname{Kr}}
\newcommand{\Na}{\operatorname{Na}}

\newcommand{\Spec}{\operatorname{Spec}}
\newcommand{\QSpec}{\operatorname{QSpec}}

\newcommand{\Max}{\operatorname{Max}}
\newcommand{\QMax}{\operatorname{QMax}}

\newtheorem{thm}{Theorem}[section]
\newtheorem{cor}[thm]{Corollary}
\newtheorem{lem}[thm]{Lemma}
\newtheorem{prop}[thm]{Proposition}

\newtheorem{exam}[thm]{Example}
\newtheorem{rem}[thm]{Remark}

\begin{document}

\bibliographystyle{amsplain}

\date{}

\author{Parviz Sahandi}

\address{Department of Mathematics, University of Tabriz, Tabriz,
Iran, and School of Mathematics, Institute for Research in
Fundamental Sciences (IPM), P.O. Box: 19395-5746, Tehran Iran}
\email{sahandi@tabrizu.ac.ir, sahandi@ipm.ir}



\thanks{Parviz Sahandi was in part supported by a grant from IPM (No.
91130030)}

\title[Pr\"{u}fer $\star$-multiplication domains]{Characterizations of graded Pr\"{u}fer $\star$-multiplication domains}

\begin{abstract} Let $R=\bigoplus_{\alpha\in\Gamma}R_{\alpha}$ be a graded
integral domain graded by an arbitrary grading torsionless monoid
$\Gamma$, and $\star$ be a semistar operation on $R$. In this paper
we define and study the graded integral domain analogue of
$\star$-Nagata and Kronecker function rings of $R$ with respect to
$\star$. We say that $R$ is a graded Pr\"{u}fer
$\star$-multiplication domain if each nonzero finitely generated
homogeneous ideal of $R$ is $\star_f$-invertible. Using
$\star$-Nagata and Kronecker function rings, we give several
different equivalent conditions for $R$ to be a graded Pr\"{u}fer
$\star$-multiplication domain. In particular we give new
characterizations for a graded integral domain, to be a P$v$MD.
\end{abstract}

\maketitle

\section{Introduction}

Let $R=\bigoplus_{\alpha\in\Gamma}R_{\alpha}$ be a graded
(commutative) integral domain graded by an arbitrary grading
torsionless monoid $\Gamma$, that is $\Gamma$ is a commutative
cancellative monoid (written additively). Let
$\langle\Gamma\rangle=\{a-b|a,b\in\Gamma\},$ be the quotient group
of $\Gamma$, which is a torsionfree abelian group.

Let $H$ be the saturated multiplicative set of nonzero homogeneous
elements of $R$. Then
$R_H=\bigoplus_{\alpha\in\langle\Gamma\rangle}(R_H)_{\alpha}$,
called the \emph{homogeneous quotient field of $R$}, is a graded
integral domain whose nonzero homogeneous elements are units. For a
fractional ideal $I$ of $R$ let $I_h$ denote the fractional ideal
generated by the set of homogeneous elements of $R$ in $I$. It is
known that if $I$ is a prime ideal, then $I_h$ is also a prime ideal
(cf. \cite[Page 124]{North}). An integral ideal $I$ of $R$ is said
to be homogeneous if $I=\bigoplus_{\alpha\in\Gamma}(I\cap
R_{\alpha})$; equivalently, if $I=I_h$. A fractional ideal $I$ of
$R$ is \emph{homogeneous} if $sI$ is an integral homogeneous ideal
of $R$ for some $s\in H$ (thus $I\subseteq R_H$). For $f\in R_H$,
let $C_R(f)$ (or simply $C(f)$) denote the fractional ideal of $R$
generated by the homogeneous components of $f$. For a fractional
ideal $I$ of $R$ with $I\subseteq R_H$, let $C(I)=\sum_{f\in
I}C(f)$. For more on graded integral domains and their divisibility
properties, see \cite{AA2, North}.

Let $R=\bigoplus_{\alpha\in\Gamma}R_{\alpha}$ and $N_v(H)=\{f\in
R|C(f)^v=R\}$. (Definitions related to the $v$-operation will be
reviewed in the sequel.) Then $N_v(H)$ is a saturated multiplicative
subset of $R$ by \cite[Lemma 1.1(2)]{AC}. The graded integral domain
analogue of the well known Nagata ring is the ring $R_{N_v(H)}$. In
\cite{AC}, Anderson and Chang, studied relationships between the
ideal-theoretic properties of $R_{N_v(H)}$ and the homogeneous
ideal-theoretic properties of $R$. For example it is shown that if
$R$ has a unit of nonzero degree, $Pic(R_{N_v(H)})=0$ and that $R$
is a P$v$MD if and only if each ideal of $R_{N_v(H)}$ is extended
from a homogeneous ideal of $R$, if and only if $R_{N_v(H)}$ is a
Pr\"{u}fer (or B\'{e}zout) domain \cite[Theorems 3.3 and 3.4]{AC}.
Also, they generalized the notion of Kronecker function ring, (for
\texttt{e.a.b.} star operations on $R$) and then showed that this
ring is a B\'{e}zout domain \cite[Theorem 3.5]{AC}. For the
definition and properties of semistar-Nagata and Kronecker function
rings of an integral domain see the interesting survey article
\cite{FL3}. Recall that the \emph{Picard group (or the ideal class
group)} of an integral domain $D$, is $Pic(D)=Inv(D)/Prin(D)$, where
$Inv(D)$ is the multiplicative group of invertible fractional ideals
of $D$, and $Prin(D)$ is the subgroup of principal fractional ideal
of $D$.

Let $R=\bigoplus_{\alpha\in\Gamma}R_{\alpha}$ be an integral domain,
and $\star$ be a semistar operation on $R$. In Section 2 of this
paper we study the homogeneous elements of $\QSpec^{\star}(R)$
denoted by $h$-$\QSpec^{\star}(R)$. We show that if $\star$ is a
finite type semistar operation on $R$ which sends homogeneous
fractional ideals to homogeneous ones, and such that
$R^{\star}\subsetneq R_H$, then each homogeneous quasi-$\star$-ideal
of $R$, is contained in a homogeneous quasi-$\star$-prime ideal of
$R$. One of key results in this paper is Proposition \ref{pp}, which
shows that if $R^{\star}\subsetneq R_H$, the $\widetilde{\star}$
sends homogeneous fractional ideals to homogeneous ones. We also
define and study the Nagata ring of $R$ with respect to $\star$. The
$\star$-Nagata ring is defined by the quotient ring
$R_{N_{\star}(H)}$, where $N_{\star}(H)=\{f\in
R|C(f)^{\star}=R^{\star}\}$. Among other things, it is shown that
$Pic(R_{N_{\star}(H)})=0$. In Section 3 we define and study the
Kronecker function ring of $R$ with respect to $\star$. The
Kronecker function ring, inspired by \cite[Theorem 5.1]{FL2}, is
defined by $\Kr(R,\star):=\{0\}\cup\{f/g|0\neq f,g\in R,\text{ and
there is }0\neq h\in R\text{ such that }
C(f)C(h)\subseteq(C(g)C(h))^{\star} \}$. It is shown that if $\star$
sends homogeneous fractional ideals to fractional ones, then
$\Kr(R,\star)$ is a B\'{e}zout domain. In Section 3 we define the
notion of graded Pr\"{u}fer $\star$-multiplication domains and give
several different equivalent conditions to be a graded P$\star$MD. A
graded integral domain $R$, is called a \emph{graded Pr\"{u}fer
$\star$-multiplication domain (graded P$\star$MD)} if every finitely
generated homogeneous ideal of $R$ is a $\star_f$-invertible, i.e.,
$(II^{-1})^{\star_f}=R^{\star}$ for each finitely generated
homogeneous ideal $I$ of $R$. Among other results we show that $R$
is a graded P$\star$MD if and only if $R_{N_{\star}(H)}$ is a
Pr\"{u}fer domain if and only if $R_{N_{\star}(H)}$ is a B\'{e}zout
domain if and only if $R_{N_{\star}(H)}=\Kr(R,\widetilde{\star})$ if
and only if $\Kr(R,\widetilde{\star})$ is a flat $R$-module.

To facilitate the reading of the paper, we review some basic facts
on semistar operations. Let $D$ be an integral domain with quotient
field $K$. Let $\overline{\mathcal{F}}(D)$ denote the set of all
nonzero $D$-submodules of $K$. Let $\mathcal{F}(D)$ be the set of
all nonzero \emph{fractional} ideals of $D$; i.e.,
$E\in\mathcal{F}(D)$ if $E\in\overline{\mathcal{F}}(D)$ and there
exists a nonzero element $r\in D$ with $rE\subseteq D$. Let $f(D)$
be the set of all nonzero finitely generated fractional ideals of
$D$. Obviously,
$f(D)\subseteq\mathcal{F}(D)\subseteq\overline{\mathcal{F}}(D)$. As
in \cite{OM}, a {\it semistar operation on} $D$ is a map
$\star:\overline{\mathcal{F}}(D)\rightarrow\overline{\mathcal{F}}(D)$,
$E\mapsto E^{\star}$, such that, for all $x\in K$, $x\neq 0$, and
for all $E, F\in\overline{\mathcal{F}}(D)$, the following three
properties hold:
\begin{itemize}
\item [$\star_1$]:  $(xE)^{\star}=xE^{\star}$;
\item [$\star_2$]:  $E\subseteq F$ implies that $E^{\star}\subseteq
F^{\star}$;
\item [$\star_3$]:  $E\subseteq E^{\star}$ and
$E^{\star\star}:=(E^{\star})^{\star}=E^{\star}$.
\end{itemize}

Let $\star$ be a semistar operation on the domain $D$. For every
$E\in\overline{\mathcal{F}}(D)$, put $E^{\star_f}:=\cup F^{\star}$,
where the union is taken over all finitely generated $F\in f(D)$
with $F\subseteq E$. It is easy to see that $\star_f$ is a semistar
operation on $D$, and ${\star_f}$ is called \emph{the semistar
operation of finite type associated to} $\star$. Note that
$(\star_f)_f=\star_f$. A semistar operation $\star$ is said to be of
\emph{finite type} if $\star=\star_f$; in particular ${\star_f}$ is
of finite type. We say that a nonzero ideal $I$ of $D$ is a
\emph{quasi-$\star$-ideal} of $D$, if $I^{\star}\cap D=I$; a
\emph{quasi-$\star$-prime} (ideal of $D$), if $I$ is a prime
quasi-$\star$-ideal of $D$; and a \emph{quasi-$\star$-maximal}
(ideal of $D$), if $I$ is maximal in the set of all proper
quasi-$\star$-ideals of $D$. Each quasi-$\star$-maximal ideal is a
prime ideal. It was shown in \cite[Lemma 4.20]{FH} that if
$D^{\star} \neq K$, then each proper quasi-$\star_f$-ideal of $D$ is
contained in a quasi-$\star_f$-maximal ideal of $D$. We denote by
$\QMax^{\star}(D)$ (resp., $\QSpec^{\star}(D)$) the set of all
quasi-$\star$-maximal ideals (resp., quasi-$\star$-prime ideals) of
$D$.

If $\star_1$ and $\star_2$ are semistar operations on $D$, one says
that $\star_1\leq\star_2$ if $E^{\star_1}\subseteq E^{\star_2}$ for
each $E\in\overline{\mathcal{F}}(D)$ (cf. \cite[page 6]{OM}). This
is equivalent to saying that
$(E^{\star_1})^{\star_2}=E^{\star_2}=(E^{\star_2})^{\star_1}$ for
each $E\in\overline{\mathcal{F}}(D)$ (cf. \cite[Lemma 16]{OM}).
Obviously, for each semistar operation $\star$ defined on $D$, we
have $\star_f\leq\star$. Let $d_D$ (or, simply, $d$) denote the
identity (semi)star operation on $D$. Clearly, $d_D\leq\star$ for
all semistar operations $\star$ on $D$.

It has become standard to say that a semistar operation $\star$ is
{\it stable} if $(E\cap F)^{\star}=E^{\star}\cap F^{\star}$ for all
$E$, $F\in \overline{\mathcal{F}}(D)$. (``Stable" has replaced the
earlier usage, ``quotient", in \cite[Definition 21]{OM}.) Given a
semistar operation $\star$ on $D$, it is possible to construct a
semistar operation $\widetilde{\star}$, which is stable and of
finite type defined as follows: for each
$E\in\overline{\mathcal{F}}(D)$,
$$
E^{\widetilde{\star}}:=\{x\in K|xJ\subseteq E,\text{ for some
}J\subseteq R, J\in f(R), J^{\star}=D^{\star}\}.
$$

It is well known that \cite[Corollary 2.7]{FH}
$$E^{\widetilde{\star}}:= \cap\{ED_P|P\in\QMax^{\star_f}(D)\}\text{, for each
}E\in\overline{\mathcal{F}}(D).$$

The most widely studied (semi)star operations on $D$ have been the
identity $d$, $v$, $t:=v_f$, and $w:=\widetilde{v}$ operations,
where $A^{v}:=(A^{-1})^{-1}$, with $A^{-1}:=(R:A):=\{x\in
K|xA\subseteq D\}$.

Let $\star$ be a semistar operation on an integral domain $D$. We
say that $\star$ is an \emph{\texttt{e.a.b.} (endlich arithmetisch
brauchbar) semistar operation} of $D$ if, for all $E, F, G\in f(D)$,
$(EF)^{\star}\subseteq(EG)^{\star}$ implies that $F^{\star}\subseteq
G^{\star}$ (\cite[Definition 2.3 and Lemma 2.7]{FL2}). We can
associate  to any semistar operation $\star$ on $D$, an
\texttt{e.a.b.} semistar operation of finite type $\star_a$ on $D$,
called the \emph{\texttt{e.a.b.} semistar operation associated to
$\star$}, defined as follows for each $F\in f(D)$ and for each $E\in
\overline{F}(D)$:
\begin{align*}
F^{\star_a}:=&\bigcup\{((FH)^{\star}:H^{\star})| H\in f(R)\},\\[1ex]
E^{\star_a}:=&\bigcup\{F^{\star_a}| F\subseteq E, F\in f(R)\}
\end{align*}
\cite[Definition 4.4 and Proposition 4.5]{FL2} (note that
$((FH)^{\star}:H^{\star})=((FH)^{\star}:H)$). It is known that
$\star_f\leq\star_a$ \cite[Proposition 4.5(3)]{FL2}. Obviously
$(\star_f)_a=\star_a$. Moreover, when $\star=\star_f$, then $\star$
is \texttt{e.a.b.} if and only if $\star=\star_a$ \cite[Proposition
4.5(5)]{FL2}.

Let $\star$ be a semistar operation on a domain $D$. Recall from
\cite{FJS} that, $D$ is called a \emph{Pr\"{u}fer
$\star$-multiplication domain} (for short, a P$\star$MD) if each
finitely generated ideal of $D$ is \emph{$\star_f$-invertible};
i.e., if $(II^{-1})^{\star_f}=D^{\star}$ for all $I\in f(D)$. When
$\star=v$, we recover the classical notion of P$v$MD; when
$\star=d_D$, the identity (semi)star operation, we recover the
notion of Pr\"{u}fer domain.

\section{Nagata ring}

Let $R=\bigoplus_{\alpha\in\Gamma}R_{\alpha}$ be a graded integral
domain, $\star$ be a semistar operation on $R$, $H$ be the set of
nonzero homogeneous elements of $R$. An overring $T$ of $R$, with
$R\subseteq T\subseteq R_H$ will be called a \emph{homogeneous
overring} if $T=\bigoplus_{\alpha\in\langle\Gamma\rangle}(T\cap
(R_H)_{\alpha})$. Thus $T$ is a graded integral domain with
$T_{\alpha}=T\cap (R_H)_{\alpha}$.

In this section we study the homogeneous elements of
$\QSpec^{\star}(R)$, denoted by $h$-$\QSpec^{\star}(R)$, and the
graded integral domain analogue of $\star$-Nagata ring. Let
$h$-$\QMax^{\star}(R)$ denote the set of ideals of $R$ which are
maximal in the set of all proper homogeneous quasi-$\star$-ideals of
$R$. The following lemma shows that, if $R^{\star}\subsetneq R_H$
and $\star=\star_f$ sends homogeneous fractional ideals to
homogeneous ones, then $h$-$\QMax^{\star_f}(R)$ is nonempty and each
proper homogeneous quasi-$\star_f$-ideal is contained in a maximal
homogeneous quasi-$\star_f$-ideal.

\begin{lem}\label{l} Let $R=\bigoplus_{\alpha\in\Gamma}R_{\alpha}$ be a
graded integral domain, $\star$ a finite type semistar operation on
$R$ which sends homogeneous fractional ideals to homogeneous ones,
and such that $R^{\star}\subsetneq R_H$. If $I$ is a proper
homogeneous quasi-$\star$-ideal of $R$, then $I$ is contained in a
proper homogeneous quasi-$\star$-prime ideal.
\end{lem}

\begin{proof} Let $X:=\{I|I$ is a homogeneous quasi-$\star$-ideal of
$R\}$. Then it is easy to see that $X$ is nonempty. Indeed, in this
case $R^{\star}$ is a homogeneous overring of $R$, and if $u\in H$
is a nonunit in $R^{\star}$, then $uR^{\star}\cap R$ is a proper
homogeneous quasi-$\star$-ideal of $R$. Also $X$ is inductive (see
proof of \cite[Lemma 4.20]{FH}). From Zorn's Lemma, we see that
every proper homogeneous quasi-$\star$-ideal of $R$ is contained in
some maximal element $Q$ of $X$.

Now we show that $Q$ is actually prime. Take $f,g\in H\backslash Q$
and suppose that $fg\in Q$. By the maximality of $Q$ we have
$(Q,f)^{\star}=R^{\star}$ (note that $(Q,f)^{\star}\cap R$ is a
homogeneous quasi-$\star$-ideal of $R$ and properly contains $Q$).
Since $\star$ is of finite type, we can find a finitely generated
ideal $J\subseteq Q$ such that $(J,f)^{\star}=R^{\star}$. Then $g\in
gR^{\star}\cap R=g(J,f)^{\star}\cap R\subseteq Q^{\star}\cap R=Q$ a
contradiction. Thus $Q$ is a prime ideal.
\end{proof}

The following example shows that we can not drop the condition that,
$\star$ sends homogeneous fractional ideals to homogeneous ones, in
the above lemma.

\begin{exam}\label{e} Let $k$ be a field and $X,Y$ be indeterminates over $k$. Let $R=k[X,Y]$,
which is a ($\mathbb{N}_0$-)graded Noetherian integral domain with
$\deg X=\deg Y=1$. Set $M:=(X,Y+1)$ which is a maximal
non-homogeneous ideal of $R$. Let $T$ be a DVR \cite{Chev}, with
maximal ideal $N$, dominating the local ring $R_M$. If $R_H\subseteq
T$, then there exists a prime ideal $P$ of $R$ such that, $P\cap
H=\emptyset$ and $N\cap R_H=PR_H$. Thus $M=N\cap R=N\cap R_H\cap
R=PR_H\cap R=P$. Hence $M\cap H=\emptyset$, which is a
contradiction, since $X\in M\cap H$. So that, $R_H\nsubseteq T$. Let
$\star$ be a semistar operation on $R$ defined by $E^{\star}=ET\cap
ER_H$ for each $E\in\overline{\mathcal{F}}(R)$. Then clearly
$\star=\star_f$ and $R^{\star}\subsetneq R_H$. If $P$ is a nonzero
prime ideal of $R$, such that $P\cap H=\emptyset$, then
$P^{\star_f}\cap R= PT\cap PR_H\cap R=PT\cap P=P$. Thus $P$ is a
quasi-$\star_f$-prime ideal. On the other hand if $P$ is any nonzero
prime ideal of $R$ such that $P\cap H\neq\emptyset$, then $PT=N^k$,
for some integer $k\geq1$. Therefore, if we assume that $P$ is a
quasi-$\star_f$-ideal of $R$, then we would have $P=PT\cap PR_H\cap
R=PT\cap R=N^k\cap R\supseteq M^k$, which implies that $P=M$. Thus
$\QSpec^{\star_f}(R)=\{M\}\cup\{P\in\Spec(R)|P\neq0$ and $P\cap
H=\emptyset\}$. Therefore by \cite[Lemma 4.1, Remark 4.5]{FH}, we
have $\QSpec^{\widetilde{\star}}(R)=\{Q\in\Spec(R)|0\neq Q\subseteq
M\}\cup \{P\in\Spec(R)|P\neq0$ and $P\cap H=\emptyset\}$. Hence in
the present example we have
$h$-$\QSpec^{\star_f}(R)=h$-$\QMax^{\star_f}(R)=\emptyset$, and
$h$-$\QSpec^{\widetilde{\star}}(R)=h$-$\QMax^{\widetilde{\star}}(R)=\{(X)\}$.
Note that in this example
$h$-$\QMax^{\widetilde{\star}}(R)\nsubseteq\QMax^{\widetilde{\star}}(R)=\QMax^{\star_f}(R)$.
\end{exam}

From now on in this paper, we are interested and consider, the
semistar operations $\star$ on $R$, such that $R^{\star}\subsetneq
R_H$ and sends homogeneous fractional ideals to homogeneous ones.
For any such semistar operation, if $I$ is a homogeneous ideal of
$R$, we have $I^{\star_f}=R^{\star}$ if and only if $I\nsubseteq Q$
for each $Q\in h$-$\QMax^{\star_f}(R)$. Also if $P$ is a
quasi-$\star$-prime ideal of $R$, then either $P_h=0$ or $P_h$ is a
quasi-$\star$-prime ideal of $R$. Indeed, if $P_h\neq0$, then
$P_h\subseteq (P_h)^{\star}\cap R\subseteq P^{\star}\cap R=P$, which
implies that $P_h=(P_h)^{\star}\cap R$, since $(P_h)^{\star}\cap R$
is a homogeneous ideal.

The following proposition is the key result in this paper.

\begin{prop}\label{pp} Let $R=\bigoplus_{\alpha\in\Gamma}R_{\alpha}$ be a
graded integral domain, and $\star$ be a semistar operation on $R$
such that $R^{\star}\subsetneq R_H$. Then, $\widetilde{\star}$ sends
homogeneous fractional ideals to homogeneous ones. In particular
$h$-$\QMax^{\widetilde{\star}}(R)\neq\emptyset$, and
$R^{\widetilde{\star}}$ is a homogeneous overring of $R$.
\end{prop}

\begin{proof} Let $E$ be a homogenous fractional ideal of $R$.
To show that $E^{\widetilde{\star}}$ is homogeneous let $f\in
E^{\widetilde{\star}}$. Then $fJ\subseteq E$ for some finitely
generated ideal $J$ of $R$ such that $J^{\star}=R^{\star}$. Suppose
that $J=(g_1,\cdots,g_n)$. Using \cite[Lemma 1.1(1)]{AC}, there is
an integer $m\geq1$ such that $C(g_i)^{m+1}C(f)=C(g_i)^mC(fg_i)$ for
all $i=1,\cdots,n$. Since $E$ is a homogeneous fractional ideal and
$fg_i\in E$, we have $C(fg_i)\subseteq E$. Thus we have
$C(g_i)^{m+1}C(f)\subseteq E$. Let
$J_0:=C(g_1)^{m+1}+\cdots+C(g_n)^{m+1}$. Thus $J_0$ is a finitely
generated homogeneous ideal of $R$ such that
$J^{\star}_0=R^{\star}$. Since $C(f)J_0\subseteq E$, $C(f)\subseteq
E^{\widetilde{\star}}$. Therefore $E^{\widetilde{\star}}$ is a
homogeneous ideal.
\end{proof}

\begin{lem}\label{homof} Let $R=\bigoplus_{\alpha\in\Gamma}R_{\alpha}$ be a
graded integral domain, $\star$ a semistar operation on $R$ which
sends homogeneous fractional ideals to homogeneous ones. Then
$\star_f$ sends homogeneous fractional ideals to homogeneous ones.
\end{lem}

\begin{proof} Let $E$ be a homogenous fractional ideal of $R$. Let
$0\neq x\in E^{\star_f}$. Then, there exists an $F\in f(R)$ such
that $F\subseteq E$ and $x\in F^{\star}$. Suppose that $F$ is
generated by $y_1,\cdots,y_n\in R_H$. Let $G$ be a homogeneous
fractional ideal of $R$, generated by homogeneous components of
$y_1,\cdots,y_n$. Note that $F\subseteq G\subseteq E$ and $x\in
G^{\star}$. Thus homogeneous components of $x$ belong to
$G^{\star}\subseteq E^{\star_f}$. This shows that $E^{\star_f}$ is
homogeneous.
\end{proof}

Note that the $v$-operation sends homogeneous fractional ideals to
homogeneous ones by \cite[Proposition 2.5]{AA2}. Using the above two
results, the $t$ and $w$-operations also, send homogeneous
fractional ideals to homogeneous ones.

It it well-known that
$\QMax^{\star_f}(R)=\QMax^{\widetilde{\star}}(R)$, see \cite[Theorem
2.16]{ACo}, for star operation case, and \cite[Corollary
3.5(2)]{FL}, in general semistar operations. Although Example
\ref{e}, shows that it may happen that $h$-$\QMax^{\star_f}(R)\neq
h$-$\QMax^{\widetilde{\star}}(R)$, we have the following proposition
whose proof is almost the same as \cite[Theorem 2.16]{AC}.

\begin{prop}\label{t=w} Let $R=\bigoplus_{\alpha\in\Gamma}R_{\alpha}$ be a
graded integral domain, $\star$ a semistar operation on $R$ such
that $R^{\star}\subsetneq R_H$, which sends homogeneous fractional
ideals to homogeneous ones. Then
$h$-$\QMax^{\star_f}(R)=h$-$\QMax^{\widetilde{\star}}(R)$.
\end{prop}

\begin{proof} Assume that $Q\in h$-$\QMax^{\star_f}(R)$.
Then since $\widetilde{\star}\leq\star_f$ by \cite[Lemma
2.7(1)]{FL}, we have $Q\subseteq Q^{\widetilde{\star}}\cap
R\subseteq Q^{\star_f}\cap R=Q$, that is $Q$ is a
quasi-$\widetilde{\star}$-ideal. Suppose that $Q\notin
h$-$\QMax^{\widetilde{\star}}(R)$. Then $Q$ is properly contained in
some $P\in h$-$\QMax^{\widetilde{\star}}(R)$. So since $Q\in
h$-$\QMax^{\star_f}(R)$, using Lemma \ref{l}, we must have
$P^{\star_f}=R^{\star}$. Thus there is some finitely generated ideal
$F\subseteq P$ such that $F^{\star}=R^{\star}$. So for any $r\in R$,
$rF\subseteq F\subseteq P$. But then, $r\in P^{\widetilde{\star}}$,
so $R\subseteq P^{\widetilde{\star}}$, which implies that
$P^{\widetilde{\star}}=R^{\widetilde{\star}}$, a contradiction.
Therefore, we must have $Q\in h$-$\QMax^{\widetilde{\star}}(R)$.

If $Q\in h$-$\QMax^{\widetilde{\star}}(R)$, then
$Q=Q^{\widetilde{\star}}\cap R\subseteq Q^{\star_f}\cap R\subseteq
R$. Suppose that $Q^{\star_f}\cap R=R$, which implies that
$Q^{\star_f}=R^{\star}$. Then there is a finitely generated ideal
$F\subseteq Q$ such that $F^{\star}=R^{\star}$. Now for any $r\in
R$, $rF\subseteq F\subseteq Q$. Therefore $R\subseteq
Q^{\widetilde{\star}}$, and so $R=Q^{\widetilde{\star}}\cap R=Q$,
which is a contradiction. So $Q^{\star_f}\cap R\subsetneq R$. Now,
since $Q^{\star_f}\cap R$ is a homogeneous quasi-$\star_f$-ideal,
there is a $P\in h$-$\QMax^{\star_f}(R)$ such that $Q\subseteq
Q^{\star_f}\cap R\subseteq P$. From the first half of the proof, we
know that $P\in h$-$\QMax^{\widetilde{\star}}(R)$. So we must have
$P=Q$. Therefore $Q\in h$-$\QMax^{\star_f}(R)$.
\end{proof}

Park in \cite[Lemma 3.4]{P}, proved that $I^w=\bigcap_{P\in
h\text{-}QMax^w(R)}IR_{H\backslash P}$ for each homogeneous ideal
$I$ of $R$.

\begin{prop}\label{tilda} Let $R=\bigoplus_{\alpha\in\Gamma}R_{\alpha}$ be a
graded integral domain, $\star$ a semistar operation on $R$ such
that $R^{\star}\subsetneq R_H$. Then
$I^{\widetilde{\star}}=\bigcap_{P\in
h\text{-}\QMax^{\widetilde{\star}}(R)}IR_{H\backslash P}$ for each
homogeneous ideal $I$ of $R$. Moreover
$I^{\widetilde{\star}}R_{H\backslash P}=IR_{H\backslash P}$ for all
homogeneous ideal $I$ of $R$ and all $P\in
h$-$\QMax^{\widetilde{\star}}(R)$.
\end{prop}

\begin{proof} By Proposition \ref{pp}, $I^{\widetilde{\star}}$ is a homogeneous
ideal. Also note that $\bigcap_{P\in
h\text{-}\QMax^{\widetilde{\star}}(R)}IR_{H\backslash P}$ is a
homogeneous ideal of $R$. Let $f\in I^{\widetilde{\star}}$ be
homogeneous. Then $fJ\subseteq I$ for some homogeneous finitely
generated ideal $J$ of $R$ such that $J^{\star}=R^{\star}$. It is
easy to see that $J^{\widetilde{\star}}=R^{\widetilde{\star}}$.
Hence we have $J\nsubseteq P$ for all $P\in
h$-$\QMax^{\widetilde{\star}}(R)$. Thus $f\in IR_{H\backslash P}$
for all $P\in h$-$\QMax^{\widetilde{\star}}(R)$. Conversely, let
$f\in \bigcap_{P\in
h\text{-}\QMax^{\widetilde{\star}}(R)}IR_{H\backslash P}$ be
homogeneous. Then $(I:f)$ is a homogeneous ideal which is not
contained in any $P\in h$-$\QMax^{\widetilde{\star}}(R)$. Therefore
$(I:f)^{\widetilde{\star}}=R^{\widetilde{\star}}$. So that there
exist a finitely generated ideal $J\subseteq(I:f)$ such that
$J^{\star}=R^{\star}$. Thus $fJ\subseteq I$, i.e., $f\in
I^{\widetilde{\star}}$. The second assertion follows from the first
one.
\end{proof}

Let $D$ be a domain with quotient field $K$, and let $X$ be an
indeterminate over $K$. For each $f\in K[X]$, we let $c_D(f)$ denote
the content of the polynomial $f$, i.e., the (fractional) ideal of
$D$ generated by the coefficients of $f$. Let $\star$ be a semistar
operation on $D$. If $N_{\star}:=\{g\in D[X]|g\neq0\text{ and
}c_D(g)^{\star}=D^{\star}\}$, then $N_{\star}=
D[X]\backslash\bigcup\{P[X]|P\in\QMax^{\star_f}(D)\}$ is a saturated
multiplicative subset of $D[X]$. The ring of fractions
$$\Na(D,\star):=D[X]_{N_{\star}}$$ is called the $\star$-{\it Nagata domain (of $D$ with respect to the
semistar operation} $\star$). When $\star=d$, the identity
(semi)star operation on $D$, then $\Na(D,d)$ coincides with the
classical Nagata domain $D(X)$ (as in, for instance \cite[page
18]{Na}, \cite[Section 33]{G} and \cite{FL}).

Let $N_{\star}(H)=\{f\in R|C(f)^{\star}=R^{\star}\}$. It is easy to
see that $N_{\star}(H)$ is a saturated multiplicative subset of $R$.
Indeed assume $f,g\in N_{\star}(H)$. Then
$C(f)^{n+1}C(g)=C(f)^nC(fg)$ for some integer $n\geq1$ by
\cite[Lemma 1.1(2)]{AC}, and $C(fg)\subseteq C(f)C(g)$. Thus $fg\in
N_{\star}(H)\Leftrightarrow C(fg)^{\star}=R^{\star}\Leftrightarrow
C(f)^{\star}=C(g)^{\star}=R^{\star}\Leftrightarrow f, g\in
N_{\star}(H)$. Also it is easy to show that
$N_{\star}(H)=N_{\star_f}(H)=N_{\widetilde{\star}}(H)$. We define
the graded integral domain analogue of $\star$-Nagata ring, by the
quotient ring $R_{N_{\star}(H)}$. When $\star=v$, $R_{N_{\star}(H)}$
was studied in \cite{AC}, denoted by $R_{N(H)}$.

\begin{lem}\label{sharp} Let $R=\bigoplus_{\alpha\in\Gamma}R_{\alpha}$ be a
graded integral domain, and $\star$ be a semistar operation on $R$
such that $R^{\star}\subsetneq R_H$, which sends homogeneous
fractional ideals to homogeneous ones.
\begin{itemize}
\item[(1)] $N_{\star}(H)=R\backslash\bigcup_{Q\in h\text{-}\QMax^{\star_f}(R)}Q$.
\item[(2)] $\Max(R_{N_{\star}(H)})=\{QR_{N_{\star}(H)}|Q\in h\text{-}\QMax^{\star_f}(R)\}$
if and only if $R$ has the property that if $I$ is a nonzero ideal
of $R$ with $C(I)^{\star}=R^{\star}$, then $I\cap
N_{\star}(H)\neq\emptyset$.
\end{itemize}
\end{lem}

\begin{proof} (1) Let $x\in R$. Then $x\in
N_{\star}(H)\Leftrightarrow C(x)^{\star}=R^{\star}\Leftrightarrow
C(x)\nsubseteq Q$ for all $Q\in
h\text{-}\QMax^{\star_f}(R)\Leftrightarrow x\notin Q$ for all $Q\in
h\text{-}\QMax^{\star_f}(R)\Leftrightarrow$ $x\in
R\backslash\bigcup_{Q\in h\text{-}\QMax^{\star_f}(R)}Q$.

(2) $(\Rightarrow)$ Let $I$ is a nonzero ideal of $R$ with
$C(I)^{\star}=R^{\star}$. Then $I\nsubseteq Q$ for all $Q\in
h\text{-}\QMax^{\star_f}(R)$, and hence
$IR_{N_{\star}(H)}=R_{N_{\star}(H)}$. Thus $I\cap
N_{\star}(H)\neq\emptyset$.

$(\Leftarrow)$ Let $I$ be a nonzero ideal of $R$ such that
$I\subseteq\bigcup_{Q\in h\text{-}\QMax^{\star_f}(R)}Q$. If
$C(I)^{\star_f}=R^{\star}$, then, by assumption, there exists an
$f\in I$ with $C(f)^{\star}=R^{\star}$. But, since
$I\subseteq\bigcup_{Q\in h\text{-}\QMax^{\star_f}(R)}Q$, we have
$f\in Q$ for some $Q\in h\text{-}\QMax^{\star_f}(R)$, a
contradiction. Thus $C(I)^{\star}\subsetneq R^{\star}$, and hence
$I\subseteq Q$ for some $Q\in h\text{-}\QMax^{\star_f}(R)$. Thus
$\{QR_{N_{\star}(H)}|Q\in h\text{-}\QMax^{\star_f}(R)\}$ is the set
of maximal ideals of $R_{N_{\star}(H)}$ by \cite[Proposition
4.8]{G}.
\end{proof}

We will say that $R$ satisfies property $(\#_{\star})$ if, for any
nonzero ideal $I$ of $R$, $C(I)^{\star}=R^{\star}$ implies that
there exists an $f\in I$ such that $C(f)^{\star}=R^{\star}$.

\begin{exam} Let $R=\bigoplus_{\alpha\in\Gamma}R_{\alpha}$ be a
graded integral domain, and let $\star$ be a semistar operation on
$R$. If $R$ contains a unit of nonzero degree, then $R$ satisfies
property $(\#_{\star})$ (see \cite[Example 1.6]{AC} for the case
$\star=t$).
\end{exam}

The next result is a generalization of the fact that
$I^{\widetilde{\star}}=I\Na(R,\star)\cap K$, where $K$ is the
quotient field of $R$ \cite[Proposition 3.4(3)]{FL}.

\begin{lem}\label{w} Let $R=\bigoplus_{\alpha\in\Gamma}R_{\alpha}$ be a
graded integral domain, and $\star$ be a semistar operation on $R$
such that $R^{\star}\subsetneq R_H$, with property $(\#_{\star})$.
Then $I^{\widetilde{\star}}=IR_{N_{\star}(H)}\cap R_H$ and
$I^{\widetilde{\star}}R_{N_{\star}(H)}=IR_{N_{\star}(H)}$ for each
homogeneous ideal $I$ of $R$. In particular $R^{\widetilde{\star}}$
is integrally closed if and only if $R_{N_{\star}(H)}$ is integrally
closed.
\end{lem}

\begin{proof} If $I^{\widetilde{\star}}=IR_{N_{\star}(H)}\cap
R_H$, then it is easy to see that
$I^{\widetilde{\star}}R_{N_{\star}(H)}=IR_{N_{\star}(H)}$. Hence it
suffices to show that $I^{\widetilde{\star}}=IR_{N_{\star}(H)}\cap
R_H$.

$(\subseteq)$ Let $f\in I^{\widetilde{\star}}(\subseteq R_H)$, and
let $J$ be a finitely generated ideal of $R$ such that
$J^{\star}=R^{\star}$ and $fJ\subseteq I$. Then
$C(J)^{\star}=R^{\star}$, and since $R$ satisfies property
$(\#_{\star})$, there exists an $h\in J$ with
$C(h)^{\star}=R^{\star}$. Hence $h\in N_{\star}(H)$ and $fh\in I$.
Thus $f\in IR_{N_{\star}(H)}\cap R_H$.

$(\supseteq)$ Let $f=\frac{g}{h}\in IR_{N_{\star}(H)}\cap R_H$,
where $g\in I$ and $h\in N_{\star}(H)$. Then $fh=g\in I$, and since
$C(h)^{m+1}C(f)=C(h)^mC(fh)$ for some integer $m\geq1$ by
\cite[Lemma 1.1(1)]{AC}, we have $fC(h)^{m+1}\subseteq
C(f)C(h)^{m+1}=C(h)^mC(fh)=C(h)^mC(g)\subseteq I$. Also note that
$(C(h)^{m+1})^{\star}=R^{\star}$, since $C(h)^{\star}=R^{\star}$.
Thus $f\in I^{\widetilde{\star}}$.

For the in particular case, assume that $R_{N_{\star}(H)}$ is
integrally closed. Using \cite[Proposition 2.1]{AA2}, $R_H$ is a
GCD-domain, hence is integrally closed. Therefore
$R^{\widetilde{\star}}=R_{N_{\star}(H)}\cap R_H$ is integrally
closed. Conversely, assume that $R^{\widetilde{\star}}$ is
integrally closed. Then $R_Q$ is integrally closed by
\cite[Proposition 3.8]{DS} for all
$Q\in\QSpec^{\widetilde{\star}}(R)$. Let $QR_{N_{\star}(H)}$ be a
maximal ideal of $R_{N_{\star}(H)}$ for some $Q\in
h$-$\QMax^{\widetilde{\star}}(R)$. Then
$(R_{N_{\star}(H)})_{QR_{N_{\star}(H)}}=R_Q$ is integrally closed.
Thus $R_{N_{\star}(H)}$ is integrally closed.
\end{proof}

\begin{lem}\label{inv} Let $R=\bigoplus_{\alpha\in\Gamma}R_{\alpha}$ be a
graded integral domain, and $\star$ be a semistar operation on $R$
such that $R^{\star}\subsetneq R_H$, with property $(\#_{\star})$.
Then for each nonzero finitely generated homogeneous ideal $I$ of
$R$, $I$ is $\star_f$-invertible if and only if, $IR_{N_{\star}(H)}$
is invertible.
\end{lem}

\begin{proof} Let $I$ be nonzero finitely
generated homogeneous ideal of $R$, such that $I$ is
$\star_f$-invertible. Let
$QR_{N_{\star}(H)}\in\Max(R_{N_{\star}(H)})$, where $Q\in
h$-$\QMax^{\widetilde{\star}}(R)$ by Lemma \ref{sharp}(2). Thus by
\cite[Theorem 2.23]{FP},
$(IR_{N_{\star}(H)})_{QR_{N_{\star}(H)}}=IR_Q$ is invertible (is
principal) in $R_Q$. Hence $IR_{N_{\star}(H)}$ is invertible by
\cite[Theorem 7.3]{G}. Conversely, assume that $I$ is finitely
generated, and $IR_{N_{\star}(H)}$ is invertible. By flatness we
have
$I^{-1}R_{N_{\star}(H)}=(R:I)R_{N_{\star}(H)}=(R_{N_{\star}(H)}:IR_{N_{\star}(H)})=(IR_{N_{\star}(H)})^{-1}$.
Therefore, $(II^{-1})R_{N_{\star}(H)}=(IR_{N_{\star}(H)})
(I^{-1}R_{N_{\star}(H)})=(IR_{N_{\star}(H)})(IR_{N_{\star}(H)})^{-1}=R_{N_{\star}(H)}$.
Hence $II^{-1}\cap N_{\star}(H)\neq\emptyset$. Let $f\in II^{-1}\cap
N_{\star}(H)$. So that
$R^{\star}=C(f)^{\star}\subseteq(II^{-1})^{\star_f}\subseteq
R^{\star}$. Thus $I$ is $\star_f$-invertible.
\end{proof}

\begin{cor}\label{C(f)} Let $R=\bigoplus_{\alpha\in\Gamma}R_{\alpha}$ be a
graded integral domain, and $\star$ be a semistar operation on $R$
such that $R^{\star}\subsetneq R_H$, with property $(\#_{\star})$
and $0\neq f\in R$. Then the following conditions are equivalent:
\begin{itemize}
\item[(1)] $C(f)$ is $\star_f$-invertible.
\item[(2)] $C(f)R_{N_{\star}(H)}$ is invertible.
\item[(3)] $C(f)R_{N_{\star}(H)}=fR_{N_{\star}(H)}$.
\end{itemize}
\end{cor}

\begin{proof} Exactly is the same as \cite[Corollary 1.9]{AC}.
\end{proof}

Let $\mathbb{Z}$ be the additive group of integers. Clearly, the
direct sum $\Gamma\oplus\mathbb{Z}$ of $\Gamma$ with $\mathbb{Z}$ is
a torsionless grading monoid. So if $y$ is an indeterminate over
$R=\bigoplus_{\alpha\in\Gamma}R_{\alpha}$, then $R[y,y^{-1}]$ is a
graded integral domain graded by $\Gamma\oplus\mathbb{Z}$. In the
following proposition we use a technique for defining semistar
operations on integral domains, due to Chang and Fontana
\cite[Theorem 2.3]{CF1}.

\begin{prop}\label{main} Let $R=\bigoplus_{\alpha\in\Gamma}R_{\alpha}$ be a graded integral domain with quotient field $K$, let $y$, $X$ be two indeterminates
over $R$ and let $\star$ be a semistar operation on $R$ such that
$R^{\star}\subsetneq R_H$. Set $T:=R[y,y^{-1}]$, $K_1:=K(y)$ and
take the following subset of $\Spec(T)$:
$$\triangle^{\star}:=\{Q\in\Spec(T)|\text{ }Q\cap R=(0)\text{ or }Q=(Q\cap R)R[y,y^{-1}]\text{ and }(Q\cap R)^{\star_f}\subsetneq R^{\star}\}.$$
Set $S^{\star}:=T[X]\backslash(\bigcup\{Q[X]
|Q\in\triangle^{\star}\})$ and:
$$E^{\star\prime}:=E[X]_{S^{\star}}\cap
K_1, \text{   for all }E\in \overline{\mathcal{F}}(T).$$

\begin{itemize}
\item[(a)] The mapping $\star\prime:
\overline{\mathcal{F}}(T)\to\overline{\mathcal{F}}(T)$, $E\mapsto
E^{\star\prime}$ is a stable semistar operation of finite type on
$T$, i.e., $\widetilde{\star\prime}=\star\prime$.

\item[(b)] $(\widetilde{\star})\prime=(\star_f)\prime=\star\prime$.

\item[(c)] $(ER[y,y^{-1}])^{\star\prime}\cap K=E^{\widetilde{\star}}$ for all $E\in\overline{\mathcal{F}}(R)$.

\item[(d)] $(ER[y,y^{-1}])^{\star\prime}=E^{\widetilde{\star}}R[y,y^{-1}]$ for all $E\in\overline{\mathcal{F}}(R)$.

\item[(e)] $T^{\star\prime}\subsetneq T_{H'}$, where $H'$ is the set of nonzero homogeneous elements of $T$, and $\star\prime$ sends
homogeneous fractional ideals to homogeneous ones.

\item[(f)] $\QMax^{\star\prime}(T)=\{Q|Q\in\Spec(T)\text{ such that
}Q\cap R=(0)\text{ and
}c_R(Q)^{\star_f}=R^{\star}\}\cup\{PR[y,y^{-1}]|P\in\QMax^{\star_f}(R)\}$.

\item[(g)] $h$-$\QMax^{\star\prime}(T)=\{PR[y,y^{-1}]|P\in h\text{-}\QMax^{\widetilde{\star}}(R)\}$.

\item[(h)] $(w_R)\prime=(t_R)\prime=(v_R)\prime=w_T$.
\end{itemize}
\end{prop}

\begin{proof} Set
$\nabla^{\star}:=\{Q\in\Spec(T)|\text{ }Q\cap R=(0)$ and
$c_D(Q)^{\star_f}=R^{\star}$ or $Q=PR[y,y^{-1}]$ and
$P\in\QMax^{\star_f}(D)\}.$ Then it is easy to see that the elements
of $\nabla^{\star}$ are the maximal elements of $\triangle^{\star}$
(see proof of \cite[Theorem 2.3]{CF1}). Thus
$$S^{\star}:=T[X]\backslash(\bigcup\{Q[X]
|Q\in\triangle^{\star}\})=T[X]\backslash(\bigcup\{Q[X]
|Q\in\nabla^{\star}\}).$$

(a) It follows from \cite[Theorem 2.1 (a) and (b)]{CF1}, that
$\star\prime$ is a stable semistar operation of finite type on $T$.

(b) Since $\QMax^{\star_f}(D)=\QMax^{\widetilde{\star}}(D)$, the
conclusion follows easily from the fact that
$S^{\widetilde{\star}}=S^{\star_f}=S^{\star}$.

(c) and (d) Exactly are the same as proof of \cite[Theorem 2.3(c)
and (d)]{CF1}.

(e) From part (d) we have
$T^{\star\prime}=R^{\widetilde{\star}}R[y,y^{-1}]\subsetneq
R_HR[y,y^{-1}]=T_{H'}$. The second assertion follows from
Proposition \ref{pp}, since $\widetilde{\star\prime}=\star\prime$ by
(a).

(f) Follows from \cite[Theorem 2.1(e)]{CF1} and the remark in the
first paragraph in the proof.

(g) Let $M\in h$-$\QMax^{\star\prime}(T)$. Since $y,y^{-1}\in T$,
clearly we have $M\cap R\neq(0)$. Then by (f), there is $P\in
\QMax^{\star_f}(R)$ such that $M\subseteq PR[y,y^{-1}]$. If $P\in
h$-$\QMax^{\widetilde{\star}}(R)$, then $M=PR[y,y^{-1}]$ and we are
done. So suppose that $P\notin h$-$\QMax^{\widetilde{\star}}(R)$.
Then note that $P_h\in h$-$\QSpec^{\widetilde{\star}}(R)$ and
$M\subseteq P_hR[y,y^{-1}]=(PR[y,y^{-1}])_h$; hence
$M=P_hR[y,y^{-1}]$, because $M$ is a homogeneous maximal
quasi-$\star\prime$-ideal. Note that in this case $P_h\in
h$-$\QMax^{\widetilde{\star}}(R)$ by \cite[Lemma 4.1, Remark
4.5]{FH}. So that $M\in\{PR[y,y^{-1}]|P\in
h\text{-}\QMax^{\widetilde{\star}}(R)\}$. The other inclusion is
trivial.

(h) Suppose that $\star_f=t$. Note that if
$M\in\QMax^{\star\prime}(T)$, and $M\cap R\neq(0)$, then, $M=(M\cap
R)[y,y^{-1}]$ and $M\cap R\in\QMax^t(R)$ (cf. \cite[Proposition
1.1]{HZ}). Moreover, if $Q\in \Spec(T)$ is such that $Q\cap R=(0)$,
then $Q$ is a quasi-$t$-maximal ideal of $T$ if and only if
$c_R(Q)^t=R$. Indeed, if $Q$ is a quasi-$t$-maximal ideal of $T$,
and $c_R(Q)^t\subsetneq R$, then there exists a quasi-$t$-maximal
ideal $P$ of $R$ such that $c_R(Q)^t\subseteq P$. Hence $Q\subseteq
P[y,y^{-1}]$, and therefore $Q=P[y,y^{-1}]$. Consequently $(0)=Q\cap
R=P[y,y^{-1}]\cap R=P$ which is a contradiction. Conversely assume
that $c_R(Q)^t=R$. Suppose $Q$ is not a quasi-$t$-maximal ideal of
$T$, and let $M$ be a quasi-$t$-maximal ideal of $T$ which contains
$Q$. Since the containment is proper, we have $M\cap R\neq(0)$. Thus
$M=(M\cap R)[y,y^{-1}]$ and $M\cap R\in\QMax^t(R)$ (cf.
\cite[Proposition 1.1]{HZ}). Since $Q\subseteq M$, $c_R(Q)$ is
contained in the quasi-$t$-ideal $M\cap R$, so that $c_R(Q)^t\neq R$
which is a contradiction. Thus we showed that
$\QMax^t(T)=\{Q|Q\in\Spec(T)\text{ such that }Q\cap R=(0)\text{ and
}c_R(Q)^{\star_f}=R^{\star}\}\cup\{PR[y,y^{-1}]|P\in\QMax^{\star_f}(R)\}=\QMax^{\star\prime}(T)$,
where the second equality is by (f). Thus using (a) and (b), we
obtain $(w_R)\prime=(t_R)\prime=(v_R)\prime=w_T$.
\end{proof}

It is known that $Pic(D(X))=0$ \cite[Theorem 2]{A}. More generally,
if $*$ is a star operation on $D$, then $Pic(\Na(D,*))=0$,
\cite[Theorem 2.14]{Kang}. Also in the graded case it is shown in
\cite[Theorem 3.3]{AC}, that $Pic(R_{N_v(H)})=0$, where
$R=\bigoplus_{\alpha\in\Gamma}R_{\alpha}$ is a graded integral
domain containing a unit of nonzero degree. We next show in general
that $Pic(R_{N_{\star}(H)})=0$.

\begin{thm}\label{p=0} Let $R=\bigoplus_{\alpha\in\Gamma}R_{\alpha}$ be a
graded integral domain with a unit of nonzero degree, and $\star$ be
a semistar operation on $R$ such that $R^{\star}\subsetneq R_H$.
Then $Pic(R_{N_{\star}(H)})=0$.
\end{thm}

\begin{proof} Let $y$ be an
indeterminate over $R$, and $T=R[y,y^{-1}]$. Using Proposition
\ref{main}(e) and (g) and Lemma \ref{sharp}, we deduce that
$\Max(T_{N_{\star\prime}(H)})=\{QT_{N_{\star\prime}(H)}|Q\in
h$-$\QMax^{\star_f}(R)\}$. Next since
$\Max((R_{N_{\star}(H)})(y))=\{P(y)|P$ is a maximal ideal of
$R_{N_{\star}(H)}\}$, \cite[Proposition 33.1]{G}, we have
$\Max((R_{N_{\star}(H)})(y))=\{(QR_{N_{\star}(H)})(y)|Q\in
h$-$\QMax^{\star_f}(R)\}$. Thus by a computation similar to the
proof of \cite[Lemma 3.2]{AC}, we obtain the equality
$T_{N_{\star\prime}(H)}=(R_{N_{\star}(H)})(y)$. The rest of the
proof is exactly the same as proof of \cite[Theorem 3.3]{AC}, using
Proposition \ref{main}.
\end{proof}

Let $D$ be a domain and $T$ an overring of $D$. Let $\star$ and
$\star'$ be semistar operations on $D$ and $T$, respectively. One
says that $T$ is \emph{$(\star,\star')$-linked to} $D$ (or that $T$
is a $(\star,\star')${\it -linked overring of} $D$) if
$$F^{\star}=D^{\star}\Rightarrow (FT)^{\star'}=T^{\star'}$$ for each nonzero finitely generated ideal $F$ of
$D$. (The preceding definition generalizes the notion of
``$t$-linked overring" which was introduced in \cite{DHLZ}.) It is
shown in \cite[Theorem 3.8]{EF}, that $T$ is a
$(\star,\star')$-linked overring of $D$ if and only if
$\Na(D,\star)\subseteq\Na(T,\star')$. We need a graded analogue of
linkedness.

Let $R=\bigoplus_{\alpha\in\Gamma}R_{\alpha}$ be a graded integral
domain, and $T$ be a homogeneous overring of $R$. Let $\star$ and
$\star'$ be semistar operations on $R$ and $T$, respectively. We say
that \emph{$T$ is homogeneously $(\star,\star')$-linked overring of
$R$} if
$$F^{\star}=D^{\star}\Rightarrow (FT)^{\star'}=T^{\star'}$$ for each nonzero homogeneous finitely generated ideal $F$ of
$R$. We say that \emph{$T$ is homogeneously $t$-linked overring of
$R$} if $T$ is homogeneously $(t,t)$-linked overring of $R$. Also it
can be seen that $T$ is homogeneously $(\star,\star')$-linked
overring of $R$ if and only if $T$ is homogeneously
$(\widetilde{\star},\widetilde{\star'})$-linked overring of $R$ (cf.
\cite[Theorem 3.8]{EF}).

\begin{exam}\label{el} Let $R=\bigoplus_{\alpha\in\Gamma}R_{\alpha}$ be a
graded integral domain, and let $\star$ be a semistar operation on
$R$ such that $R^{\star}\subsetneq R_H$. Let $P\in
h$-$\QSpec^{\widetilde{\star}}(R)$. Then, $R_{H\backslash P}$ is a
homogeneously $(\star,\star')$-linked overring of $R$, for all
semistar operation $\star'$ on $R_{H\backslash P}$. Indeed assume
that $F$ is a nonzero finitely generated homogeneous ideal of $R$
such that $F^{\star}=R^{\star}$. Then we have
$F^{\widetilde{\star}}=R^{\widetilde{\star}}$. Thus using
Proposition \ref{tilda}, we have $FR_{H\backslash
P}=F^{\widetilde{\star}}R_{H\backslash
P}=R^{\widetilde{\star}}R_{H\backslash P}=R_{H\backslash P}$.
\end{exam}

\begin{lem}\label{link} Let $R=\bigoplus_{\alpha\in\Gamma}R_{\alpha}$ be a
graded integral domain with a unit of nonzero degree, and let $T$ be
a homogeneous overring of $R$. Let $\star$ (resp. $\star'$) be a
semistar operation on $R$ (resp. on $T$). Then, $T$ is a
homogeneously $(\star,\star')$-linked overring of $R$ if and only if
$R_{N_{\star}(H)}\subseteq T_{N_{\star'}(H)}$.
\end{lem}

\begin{proof} Let $f\in R$ such that $C_R(f)^{\star}=R^{\star}$.
Then by assumption $C_T(f)^{\star'}=(C_R(f)T)^{\star'}=R^{\star'}$.
Hence $R_{N_{\star}(H)}\subseteq T_{N_{\star'}(H)}$. Conversely let
$F$ be a nonzero homogeneous finitely generated ideal of $R$ such
that $F^{\star}=R^{\star}$. Since $R$ has a unit of nonzero degree
we can choose an element $f\in R$ such that $C_R(f)=F$. From the
fact that $C_R(f)^{\star}=R^{\star}$, we have that $f$ is a unit in
$R_{N_{\star}(H)}$ and so by assumption, $f$ is a unit in
$T_{N_{\star'}(H)}$. This implies that
$C_T(f)^{\star'}=(C_R(f)T)^{\star'}=T^{\star'}$, i.e.,
$(FT)^{\star'}=T^{\star'}$.
\end{proof}

\section{Kronecker function ring}

Let $R=\bigoplus_{\alpha\in\Gamma}R_{\alpha}$ be a graded integral
domain, $*$ an \texttt{e.a.b.} star operation on $R$. The graded
analogue of the well known Kronecker function ring (see
\cite[Theorem 32.7]{G}) of $R$ with respect to $*$ is defined by
$$
\Kr(R,*):=\left\{\frac{f}{g}\bigg| \begin{array}{l} f,g\in R,\text{
}g\neq0,\text{ and }C(f)\subseteq C(g)^* \end{array} \right\}
$$
in \cite{AC}. The following lemma is proved in \cite[Theorems 2.9
and 3.5]{AC}, for an \texttt{e.a.b.} star operation $*$. We need to
state it for \texttt{e.a.b.} semistar operations. Since the proof is
exactly the same as star operation case, we omit the proof.

\begin{lem}\label{Kr} Let $R=\bigoplus_{\alpha\in\Gamma}R_{\alpha}$ be a
graded integral domain, $\star$ an \texttt{e.a.b.} semistar
operation on $R$, and
$$
\Kr(R,\star):=\left\{\frac{f}{g}\bigg| \begin{array}{l} f,g\in
R,\text{ }g\neq0,\text{ and }C(f)\subseteq C(g)^{\star} \end{array}
\right\}.
$$
Then
\begin{itemize}
\item[(1)] $\Kr(R,\star)$ is an integral domain.
\end{itemize}
In addition, if $R$ has a unit of nonzero degree, then,
\begin{itemize}
\item[(2)] $\Kr(R,\star)$ is a B\'{e}zout domain.
\item[(3)] $I\Kr(R,\star)\cap R_H=I^{\star}$ for every nonzero finitely
generated homogeneous ideal $I$ of $R$.
\end{itemize}
\end{lem}

Inspired by the work of Fontana and Loper in \cite{FL2}, we can
generalize this definition of $\Kr(R,\star)$ to all semistar
operations on $R$ which send homogeneous fractional ideals, to
homogeneous ones, provided that $R$ has a unit of nonzero degree.
Before doing that we need a lemma.

\begin{lem}\label{o} Let $R=\bigoplus_{\alpha\in\Gamma}R_{\alpha}$ be a
graded integral domain, $\star$ a semistar operation on $R$ which
sends homogeneous fractional ideals to homogeneous ones. Suppose
that $a\in R$ is homogeneous and $B, F\in f(R)$, with $B$
homogeneous and $F\subseteq R_H$, such that $aF\subseteq
(BF)^{\star}$. Then there exists a homogeneous $T\in f(R)$ such that
$aT\subseteq (BT)^{\star}$.
\end{lem}

\begin{proof} Suppose that $F$ is generated
by $y_1,\cdots,y_n\in R_H$. Let $y_i=\sum t_{ij}$ be the
decomposition of $y_i$ to homogeneous elements for $i=1,\cdots,n$.
Then $ay_i\in(BF)^{\star}=(\sum y_iB)^{\star}\subseteq(\sum
t_{ij}B)^{\star}$. Since $(\sum t_{ij}B)^{\star}$ is homogeneous we
have $at_{ij}\in(\sum t_{ij}B)^{\star}$. Let $T$ be the fractional
ideal of $R$, generated by all homogeneous elements $t_{ij}$. So
that $aT\subseteq (BT)^{\star}$ and $T\in f(R)$ is homogeneous.
\end{proof}

\begin{thm}\label{p} Let $R=\bigoplus_{\alpha\in\Gamma}R_{\alpha}$ be a
graded integral domain with a unit of nonzero degree, $\star$ a
semistar operation on $R$ which sends homogeneous fractional ideals
to homogeneous ones, and
$$
\Kr(R,\star):=\left\{\frac{f}{g}\bigg| \begin{array}{l} f,g\in R,
g\neq0,\text{ and there is }0\neq h\in R
\\\text{ such that } C(f)C(h)\subseteq(C(g)C(h))^{\star} \end{array} \right\}.
$$
Then
\begin{itemize}
\item[(1)] $\Kr(R,\star)=\Kr(R,\star_a)$.
\item[(2)] $\Kr(R,\star)$ is a B\'{e}zout domain.
\item[(3)] $I\Kr(R,\star)\cap R_H=I^{\star_a}$ for every nonzero finitely
generated homogeneous ideal $I$ of $R$.
\item[(4)] If $f,g\in R$ are nonzero such that $C(f+g)^{\star}=(C(f)+C(g))^{\star}$,
then $(f,g)\Kr(R,\star)=(f+g)\Kr(R,\star)$. In particular,
$f\Kr(R,\star)=C(f)\Kr(R,\star)$ for all $f\in R$.
\end{itemize}
\end{thm}

\begin{proof} It it clear from the definition that $\Kr(R,\star)=\Kr(R,\star_f)$. Thus
using Lemma \ref{homof}, we can assume, without loss of generality,
that $\star$ is a semistar operation of finite type.

Parts (2) and (3) are direct consequences of (1) using Lemma
\ref{Kr}. For the proof of (1) we have two cases:

{\bf Case 1:} Assume that $\star$ is an \texttt{e.a.b.} semistar
operation of finite type. In this case, for $f,g,h\in
R\backslash\{0\}$ we have
$$
C(f)C(h)\subseteq(C(g)C(h))^{\star}\Leftrightarrow C(f)\subseteq
C(g)^{\star}.
$$
Therefore $\Kr(R,\star)$ -as defined in this theorem- coincides with
$\Kr(R,\star)$ of an \texttt{e.a.b.} semistar operation $\star$, as
defined in Lemma \ref{Kr}. Also in this case $\star=\star_a$ by
\cite[Proposition 4.5(5)]{FL2}. Hence in this case (1) is true.

{\bf Case 2:} General case. Let $\star$ be a semistar operation of
finite type on $R$. By definition it is easy to see that, given two
semistar operations on $R$ with $\star_1\leq\star_2$, then
$\Kr(R,\star_1)\subseteq\Kr(R,\star_2)$. Using \cite[Proposition
4.5(3)]{FL2} we have $\star\leq\star_a$. Therefore
$\Kr(R,\star)\subseteq\Kr(R,\star_a)$. Conversely let
$f/g\in\Kr(R,\star_a)$. Then, by Case 1, $C(f)\subseteq
C(g)^{\star_a}$. Set $A:=C(f)$ and $B:=C(g)$. Then $A\subseteq
B^{\star_a}=\bigcup\{((BH)^{\star}:H)| H\in f(R)\}$. Suppose that
$A$ is generated by homogeneous elements $x_1,\cdots,x_n\in R$. Then
there is $H_i\in f(R)$, such that $x_iH_i\subseteq (BH_i)^{\star}$
for $i=1,\cdots,n$. Choose $0\neq r_i\in R$ such that
$F_i=r_iH_i\subseteq R$. Thus $x_iF_i\subseteq (BF_i)^{\star}$.
Therefore Lemma \ref{o} gives a homogeneous $T_i\in f(R)$ such that
$x_iT_i\subseteq (BT_i)^{\star}$. Now set $T:=T_1T_2\cdots T_n$
which is a finitely generated homogeneous fractional ideal of $R$
such that $AT\subseteq(BT)^{\star}$. Now since $R$ has a unit of
nonzero degree, we can find an element $h\in R$ such that $C(h)=T$.
Then $C(f)C(h)\subseteq(C(g)C(h))^{\star}$. This means that
$f/g\in\Kr(R,\star)$ to complete the proof of (1).

The proof of (4) is exactly the same as \cite[Theorem 2.9(3)]{AC}.
\end{proof}

\section{Graded P$\star$MDs}

Let $R=\bigoplus_{\alpha\in\Gamma}R_{\alpha}$ be a graded integral
domain, $\star$ be a semistar operation on $R$, $H$ be the set of
nonzero homogeneous elements of $R$, and $N_{\star}(H)=\{f\in
R|C(f)^{\star}=R^{\star}\}$. In this section we define the notion of
graded Pr\"{u}fer $\star$-multiplication domain (graded P$\star$MD
for short) and give several characterization of it.

We say that a graded integral domain
$R=\bigoplus_{\alpha\in\Gamma}R_{\alpha}$ with a semistar operation
$\star$, is a \emph{graded Pr\"{u}fer $\star$-multiplication domain
(graded P$\star$MD)} if every nonzero finitely generated homogeneous
ideal of $R$ is a $\star_f$-invertible, i.e.,
$(II^{-1})^{\star_f}=R^{\star}$ for every nonzero finitely generated
homogeneous ideal $I$ of $R$. It is easy to see that a graded
P$\star$MD is the same as a graded P$\star_f$MD by definition, and
is the same as a graded P$\widetilde{\star}$MD by \cite[Proposition
2.18]{FP}. When $\star=v$ we recover the classical notion of a
\emph{graded Pr\"{u}fer $v$-multiplication domain (graded P$v$MD)}
\cite{AA}. It is known that $R$ is a graded P$v$MD if and only if
$R$ is a P$v$MD \cite[Theorem 6.4]{AA}.

Also when $\star=d$, a graded P$d$MD is called a \emph{graded
Pr\"{u}fer domain} \cite{AC}. It is clear that every graded
Pr\"{u}fer domain is a graded P$v$MD and hence a P$v$MD. In
particular every graded Pr\"{u}fer domain is an integrally closed
domain. Although $R$ is a graded P$v$MD if and only if $R$ is a
P$v$MD, Anderson and Chang in \cite[Example 3.6]{AC} provided an
example of a graded Pr\"{u}fer domain which is not Pr\"{u}fer. It is
known that if $A, B, C$ are ideals of an integral domain $D$, then
$(A+B)(A+C)(B+C)=(A+B+C)(AB+AC+BC)$. Thus
$R=\bigoplus_{\alpha\in\Gamma}R_{\alpha}$ is a graded Pr\"{u}fer
domain if and only if every nonzero ideal of $R$ generated by two
homogeneous elements is invertible. We use this result in this
section without comments.

The following proposition is inspired by \cite[Theorem 24.3]{G}.

\begin{prop}\label{grp} Let $R=\bigoplus_{\alpha\in\Gamma}R_{\alpha}$ be a
graded integral domain. Then the following conditions are
equivalent:
\begin{itemize}
\item[(1)] $R$ is a graded Pr\"{u}fer domain.
\item[(2)] Each finitely generated nonzero homogeneous ideal of $R$
is a cancelation ideal.
\item[(3)] If $A,B, C$ are finitely generated homogeneous ideals of $R$ such
that $AB=AC$ and $A$ is nonzero, then $B=C$.
\item[(4)] $R$ is integrally closed and there is a positive integer
$n>1$ such that $(a,b)^n=(a^n,b^n)$ for each $a, b\in H$.
\item[(5)] $R$ is integrally closed and there exists an integer
$n>1$ such that $a^{n-1}b\in(a^n,b^n)$ for each $a, b\in H$.
\end{itemize}
\end{prop}

\begin{proof} The implications $(1)\Rightarrow(2)\Rightarrow(3)$ and
$(4)\Rightarrow(5)$ are clear.

$(3)\Rightarrow(4)$ By the same argument as in the proof of part
$(2)\Rightarrow(3)$, in \cite[Proposition 24.1]{G}, we have that $R$
is integrally closed in $R_H$. Therefore by \cite[Proposition
5.4]{AA2}, $R$ is integrally closed. Now if $a, b\in H$ we have
$(a,b)^3=(a,b)(a^2,b^2)$. Thus by (3) we obtain that
$(a,b)^2=(a^2,b^2)$.

$(5)\Rightarrow(1)$ If (5) holds then \cite[Proposition 24.2]{G},
implies that each nonzero homogeneous ideal generated by two
homogeneous elements is invertible. Therefore $R$ is a graded
Pr\"{u}fer domain.
\end{proof}

The ungraded version of the following theorem is due to Gilmer (see
\cite[Corollary 28.5]{G}).

\begin{thm}\label{cp} Let $R=\bigoplus_{\alpha\in\Gamma}R_{\alpha}$ be a
graded integral domain with a unit of nonzero degree. Then $R$ is a
graded Pr\"{u}fer domain if and only if $C(f)C(g)=C(fg)$ for all $f,
g\in R_H$.
\end{thm}

\begin{proof} $(\Rightarrow)$ Let $f, g\in R_H$. Then by \cite[Lemma
1.1(1)]{AC}, there exists some positive integer $n$ such that
$C(f)^{n+1}C(g)=C(f)^nC(fg)$. Now since $R$ is a graded Pr\"{u}fer
domain, the homogeneous fractional ideal $C(f)^n$ is invertible.
Thus $C(f)C(g)=C(fg)$ for all $f, g\in R_H$.

$(\Leftarrow)$ Let $\alpha\in H$ be a unit of nonzero degree. Assume
that $C(f)C(g)=C(fg)$ for all $f, g\in R_H$. Hence $R$ is integrally
closed by \cite[Theorem 3.7]{AA}. Now let $a, b\in H$ be arbitrary.
We can choose a positive integer $n$ such that
$\deg(a)\neq\deg(\alpha^nb)$. So that $C(a+\alpha^nb)=(a,b)$. Hence,
since $(a+\alpha^nb)(a-\alpha^nb)=a^2-(\alpha^nb)^2$, we have
$(a,b)(a,-b)=(a^2,-b^2)$. Consequently $(a,b)^2=(a^2,b^2)$. Thus by
Proposition \ref{grp}, we see that $R$ is a graded Pr\"{u}fer
domain.
\end{proof}

\begin{lem}\label{lo} Let $R=\bigoplus_{\alpha\in\Gamma}R_{\alpha}$ be a
graded integral domain and $P$ be a homogeneous prime ideal. Then,
the following statements are equivalent:
\begin{itemize}
\item[(1)] $R_{H\backslash P}$ is a graded Pr\"{u}fer domain
\item[(2)] $R_P$ is a valuation domain.
\item[(3)] For each nonzero homogeneous $u\in R_H$, $u$ or
$u^{-1}$ is in $R_{H\backslash P}$.
\end{itemize}
\end{lem}

\begin{proof} $(1)\Rightarrow(2)$ Suppose that $R_{H\backslash P}$ is a graded
Pr\"{u}fer domain. In particular $R_{H\backslash P}$ is a (graded)
P$v$MD and each nonzero homogeneous ideal of $R_{H\backslash P}$ is
a $t$-ideal. So that $h$-$\QMax^t(R_{H\backslash
P})=\{PR_{H\backslash P}\}$. Thus by \cite[Lemma 2.7]{CKL}, we see
that $(R_{H\backslash P})_{PR_{H\backslash P}}=R_P$ is a valuation
domain.

$(2)\Rightarrow(3)$ Let $0\neq u\in R_H$. Thus by the hypothesis $u$
or $u^{-1}$ is in $R_P$. Thus $u$ or $u^{-1}$ is in $R_{H\backslash
P}$.

$(3)\Rightarrow(1)$ Let $I, J$ be two nonzero homogeneous ideals of
$R_{H\backslash P}$ and assume that $I\nsubseteq J$. So there is a
homogeneous element $a\in I\backslash J$. For each $b\in J$, we have
$\frac{a}{b}\notin R_{H\backslash P}$, since otherwise we have
$a=(\frac{a}{b})b\in J$. Thus by the hypothesis $\frac{b}{a}\in
R_{H\backslash P}$. Hence $b=(\frac{b}{a})a\in I$. Thus we showed
that $J\subseteq I$, and so every two homogeneous ideal are
comparable.

Now Let $(a,b)$ be an ideal generated by two homogeneous elements of
$R_{H\backslash P}$. Now by the first paragraph $(a,b)=(a)$ or
$(a,b)=(b)$. Thus $(a,b)$ is invertible. Hence $R_{H\backslash P}$
is a graded Pr\"{u}fer domain.
\end{proof}

\begin{thm}\label{loc} Let $R=\bigoplus_{\alpha\in\Gamma}R_{\alpha}$ be a
graded integral domain, and $\star$ be a semistar operation on $R$
such that $R^{\star}\subsetneq R_H$. Then, the following statements
are equivalent:
\begin{itemize}
\item[(1)] $R$ is a graded P$\star$MD.
\item[(2)] $R_{H\backslash P}$ is a graded Pr\"{u}fer domain for
each $P\in h$-$\QSpec^{\widetilde{\star}}(R)$.
\item[(3)] $R_{H\backslash P}$ is a graded Pr\"{u}fer domain for
each $P\in h$-$\QMax^{\widetilde{\star}}(R)$.
\item[(4)] $R_P$ is a valuation domain for
each $P\in h$-$\QSpec^{\widetilde{\star}}(R)$.
\item[(5)] $R_P$ is a valuation domain for
each $P\in h$-$\QMax^{\widetilde{\star}}(R)$.
\end{itemize}
\end{thm}

\begin{proof} $(2)\Rightarrow(3)$ is trivial, and,
$(2)\Leftrightarrow(4)$ and $(3)\Leftrightarrow(5)$, follow from
Lemma \ref{lo}.

$(1)\Rightarrow(2)$ Let $I$ be a nonzero finitely generated
homogeneous ideal of $R$. Then $I$ is
$\widetilde{\star}$-invertible. Therefore, for each $P\in
h$-$\QSpec^{\widetilde{\star}}(R)$, since $II^{-1}\nsubseteq P$, we
have $R_{H\backslash P}=(II^{-1})R_{H\backslash P}=IR_{H\backslash
P}I^{-1}R_{H\backslash P}=(IR_{H\backslash P})(IR_{H\backslash
P})^{-1}$. So that $IR_{H\backslash P}$ is invertible. Thus
$R_{H\backslash P}$ is a graded Pr\"{u}fer domain for each $P\in
h$-$\QSpec^{\widetilde{\star}}(R)$.

$(3)\Rightarrow(1)$ Let $I$ be a nonzero finitely generated
homogeneous ideal of $R$. Suppose that $I$ is not
$\widetilde{\star}$-invertible. Hence there exists $P\in
h$-$\QMax^{\widetilde{\star}}(R)$ such that $II^{-1}\subseteq P$.
Thus $R_{H\backslash P}=(IR_{H\backslash P})(IR_{H\backslash
P})^{-1}=II^{-1}R_{H\backslash P}\subseteq PR_{H\backslash P}$,
which is a contradiction. So that $II^{-1}\nsubseteq P$ for each
$P\in h$-$\QMax^{\widetilde{\star}}(R)$. Therefore
$(II^{-1})^{\widetilde{\star}}=R^{\widetilde{\star}}$, that is $I$
is $\widetilde{\star}$-invertible, and hence $R$ is a graded
P$\star$MD.
\end{proof}

The ungraded version of the following theorem is due to Chang in the
star operation case \cite[Theorem 3.7]{Ch}, and is due to Anderson,
Fontana, and Zafrullah in the case of semistar operations
\cite[Theorem 1.1]{AFZ}.

\begin{thm}\label{hhh} Let $R=\bigoplus_{\alpha\in\Gamma}R_{\alpha}$ be a
graded integral domain with a unit of nonzero degree, and $\star$ be
a semistar operation on $R$ such that $R^{\star}\subsetneq R_H$.
Then $R$ is a graded P$\star$MD if and only if
$(C(f)C(g))^{\widetilde{\star}}=C(fg)^{\widetilde{\star}}$ for all
$f, g\in R_H$.
\end{thm}

\begin{proof} $(\Rightarrow)$ Let $f, g\in R_H$. Choose a positive integer $n$ such that
$C(f)^{n+1}C(g)=C(f)^nC(fg)$ by \cite[Lemma 1.1(1)]{AC}. Thus
$(C(f)^{n+1}C(g))^{\widetilde{\star}}=(C(f)^nC(fg))^{\widetilde{\star}}$.
Since $R$ is a graded P$\star$MD, the homogeneous fractional ideal
$C(f)^n$ is ${\widetilde{\star}}$-invertible. Thus
$(C(f)C(g))^{\widetilde{\star}}=C(fg)^{\widetilde{\star}}$ for all
$f, g\in R_H$.

$(\Leftarrow)$ Assume that
$(C(f)C(g))^{\widetilde{\star}}=C(fg)^{\widetilde{\star}}$ for all
$f, g\in R_H$. Let $P\in h$-$\QMax^{\widetilde{\star}}(R)$. Then
using Proposition \ref{tilda}, we have $C(f)R_{H\backslash
P}C(g)R_{H\backslash P}=C(f)C(g)R_{H\backslash
P}=(C(f)C(g))^{\widetilde{\star}}R_{H\backslash
P}=C(fg)^{\widetilde{\star}}R_{H\backslash P}=C(fg)R_{H\backslash
P}$. Since $R_{H\backslash P}$ has a unit of nonzero degree, Theorem
\ref{cp} shows that $R_{H\backslash P}$ is a graded Pr\"{u}fer
domain. Now Theorem \ref{loc}, implies that $R$ is a graded
P$\star$MD.
\end{proof}

We now recall the notion of $\star$-valuation overring (a notion due
essentially to P. Jaffard \cite[page 46]{Jaf}). For a domain $D$ and
a semistar operation $\star$ on $D$, we say that a valuation
overring $V$ of $D$ is a \emph{$\star$-valuation overring of $D$}
provided $F^{\star}\subseteq FV$, for each $F\in f(D)$.

\begin{rem}\label{r} (1) Let $\star$ be a semistar operation on a graded integral domain
$R=\bigoplus_{\alpha\in\Gamma}R_{\alpha}$. Recall that for each
$F\in f(R)$ we have
$$
F^{\star_a}=\bigcap\{FV|V\text{ is a }\star\text{-valuation overring
of }R\},
$$
by \cite[Propositions 3.3 and 3.4 and Theorem 3.5]{FL1}.

(2) We have $N_{\star}(H)=N_{\widetilde{\star}_a}(H)$. Indeed, since
$\widetilde{\star}\leq\widetilde{\star}_a$ by \cite[Proposition
4.5]{FL2}, we have $N_{\star}(H)=N_{\widetilde{\star}}(H)\subseteq
N_{\widetilde{\star}_a}(H)$. Now if $f\in R\backslash N_{\star}(H)$
then, $C(f)^{\widetilde{\star}}\subsetneq R^{\widetilde{\star}}$.
Thus there is a homogeneous quasi-$\widetilde{\star}$-prime ideal
$P$ of $R$ such that $C(f)\subseteq P$. Let $V$ be a valuation
domain dominating $R_P$ with maximal ideal $M$ \cite[Corollary
19.7]{G}. Therefore $V$ is a $\widetilde{\star}$-valuation overring
of $R$ by \cite[Theorem 3.9]{FL}, and $C(f)V\subseteq M$; so
$C(f)^{(\widetilde{\star})_a}\subsetneq R^{(\widetilde{\star})_a}$
and $f\notin N_{\widetilde{\star}_a}(H)$. Thus we obtain that
$N_{\star}(H)=N_{\widetilde{\star}_a}(H)$.
\end{rem}

In the following theorem we generalize a characterization of P$v$MDs
proved by Arnold and Brewer \cite[Theorem 3]{AB}. It also
generalizes \cite[Theorem 3.7]{Ch}, \cite[Theorems 3.4 and 3.5]{AC},
and \cite[Theorem 3.1]{FJS}.

\begin{thm}\label{NKP1} Let $R=\bigoplus_{\alpha\in\Gamma}R_{\alpha}$ be a
graded integral domain with a unit of nonzero degree, and $\star$ be
a semistar operation on $R$ such that $R^{\star}\subsetneq R_H$.
Then, the following statements are equivalent:
\begin{itemize}
\item[(1)] $R$ is a graded P$\star$MD.
\item[(2)] Every ideal of $R_{N_{\star}(H)}$ is extended from a homogeneous ideal of $R$.
\item[(3)] Every principal ideal of $R_{N_{\star}(H)}$ is extended from a homogeneous ideal of $R$.
\item[(4)] $R_{N_{\star}(H)}$ is a Pr\"{u}fer domain.
\item[(5)] $R_{N_{\star}(H)}$ is a B\'{e}zout domain.
\item[(6)] $R_{N_{\star}(H)}=\Kr(R,\widetilde{\star})$.
\item[(7)] $\Kr(R,\widetilde{\star})$ is a quotient ring of $R$.
\item[(8)] $\Kr(R,\widetilde{\star})$ is a flat $R$-module.
\item[(9)] $I^{\widetilde{\star}}=I^{\widetilde{\star}_a}$ for each nonzero homogeneous finitely
generated ideal of $R$.
\end{itemize}
In particular if $R$ is a graded P$\star$MD, then
$R^{\widetilde{\star}}$ is integrally closed.
\end{thm}

\begin{proof} By Proposition \ref{pp} and Theorem
\ref{p}, we have $\Kr(R,\widetilde{\star})$ is well-defined and is a
B\'{e}zout domain.

$(1)\Rightarrow(2)$ Let $0\neq f\in R$. Then $C(f)$ is
$\widetilde{\star}$-invertible, because $R$ is a graded P$\star$MD,
and thus $fR_{N_{\star}(H)}=C(f)R_{N_{\star}(H)}$ by Corollary
\ref{C(f)}. Hence if $A$ is an ideal of $R_{N_{\star}(H)}$, then
$A=IR_{N_{\star}(H)}$ for some ideal $I$ of $R$, and thus
$A=(\sum_{f\in I}C(f))R_{N_{\star}(H)}$.

$(2)\Rightarrow(3)$ Clear.

$(3)\Rightarrow(1)$ Is the same as part $(3)\Rightarrow(1)$ in
\cite[Theorem 3.4]{AC}.

$(1)\Rightarrow(4)$ Let $A$ be a nonzero finitely generated ideal of
$R_{N_{\star}(H)}$. Then by Corollary \ref{C(f)},
$A=IR_{N_{\star}(H)}$ for some nonzero finitely generated
homogeneous ideal $I$ of $R$. Since $R$ is a graded P$\star$MD, $I$
is $\widetilde{\star}$-invertible, and thus $A=IR_{N_{\star}(H)}$ is
invertible by Lemma \ref{inv}.

$(4)\Rightarrow(5)$ Follows from Theorem \ref{p=0}.

$(5)\Rightarrow(6)$ Clearly
$R_{N_{\star}(H)}\subseteq\Kr(R,\widetilde{\star})$. Since
$R_{N_{\star}(H)}$ is a B\'{e}zout domain, then
$\Kr(R,\widetilde{\star})$ is a quotient ring of $R_{N_{\star}(H)}$,
by \cite[Proposition 27.3]{G}. If $Q\in
h$-$\QMax^{\widetilde{\star}}(R)$, then
$Q\Kr(R,\widetilde{\star})\subsetneq\Kr(R,\widetilde{\star})$.
Otherwise $Q\Kr(R,\widetilde{\star})=\Kr(R,\widetilde{\star})$, and
hence there is an element $f\in Q$, such that
$f\Kr(R,\widetilde{\star})=\Kr(R,\widetilde{\star})$. Thus
$\frac{1}{f}\in\Kr(R,\widetilde{\star})$. Therefore $R=C(1)\subseteq
C(f)^{(\widetilde{\star})_a}\subseteq R^{(\widetilde{\star})_a}$, so
that $C(f)^{(\widetilde{\star})_a}=R^{(\widetilde{\star})_a}$. Hence
$f\in N_{(\widetilde{\star})_a}(H)=N_{\star}(H)$ by Remark
\ref{r}(2). This means that
$Q^{\widetilde{\star}}=R^{\widetilde{\star}}$, a contradiction. Thus
$Q\Kr(R,\widetilde{\star})\subsetneq\Kr(R,\widetilde{\star})$, and
so there is a maximal ideal $M$ of $\Kr(R,\widetilde{\star})$ such
that $Q\Kr(R,\star)\subseteq M$. Hence $M\cap
R_{N_{\star}(H)}=QR_{N_{\star}(H)}$, by Lemma \ref{sharp}.
Consequently $R_Q\subseteq \Kr(R,\widetilde{\star})_M$, and since
$R_Q$ is a valuation domain, we have
$R_Q=\Kr(R,\widetilde{\star})_M$. Therefore
$R_{N_{\star}(H)}=\bigcap_{Q\in
h\text{-}\QMax^{\widetilde{\star}}(R)}R_Q\supseteq\bigcap_{M\in\Max(\Kr(R,\widetilde{\star}))}\Kr(R,\widetilde{\star})_M$.
Hence $R_{N_{\star}(H)}=\Kr(R,\widetilde{\star})$.

$(6)\Rightarrow(7)$ and $(7)\Rightarrow(8)$ are clear.

$(8)\Rightarrow(6)$ Recall that an overring $T$ of an integral
domain $S$ is a flat $S$-module if and only if $T_M=S_{M\cap S}$ for
all $M\in\Max(T)$ by \cite[Theorem 2]{R}.

Let $A$ be an ideal of $R$ such that
$A\Kr(R,\widetilde{\star})=\Kr(R,\widetilde{\star})$. Then there
exists an element $f\in A$ such that
$f\Kr(R,\widetilde{\star})=\Kr(R,\widetilde{\star})$ using Theorem
\ref{p}; so
$\frac{1}{f}\in\Kr(R,\widetilde{\star})=\Kr(R,\widetilde{\star}_a)$.
Thus $R=C(1)\subseteq C(f)^{\widetilde{\star}_a}\subseteq
R^{\widetilde{\star}_a}$, and so
$C(f)^{\widetilde{\star}_a}=R^{\widetilde{\star}_a}$. Hence
$C(f)^{\widetilde{\star}}=R^{\widetilde{\star}}$. Therefore $f\in
A\cap N_{\star}(H)\neq\emptyset$. Hence, if $P_0$ is a homogeneous
maximal quasi-$\widetilde{\star}$-ideal of $R$, then
$P_0\Kr(R,\widetilde{\star})\subsetneq\Kr(R,\widetilde{\star})$, and
since $P_0R_{N_{\star}(H)}$ is a maximal ideal of
$R_{N_{\star}(H)}$, there is a maximal ideal $M_0$ of
$\Kr(R,\widetilde{\star})$ such that $M_0\cap R=(M_0\cap
R_{N_{\star}(H)})\cap R=P_0R_{N_{\star}(H)}\cap R=P_0$. Thus by (8),
$\Kr(R,w)_{M_0}=R_{P_0}=(R_{N(H)})_{P_0R_{N(H)}}$.

Let $M_1$ be a maximal ideal of $\Kr(R,\widetilde{\star})$, and let
$P_1$ be a homogeneous maximal quasi-$\widetilde{\star}$-ideal of
$R$ such that $M_1\cap R_{N_{\star}(H)}\subseteq
P_1R_{N_{\star}(H)}$. By the above paragraph, there is a maximal
ideal $M_2$ of $\Kr(R,\widetilde{\star})$ such that
$\Kr(R,\widetilde{\star})_{M_2}=(R_{N_{\star}(H)})_{P_1R_{N_{\star}(H)}}$.
Note that
$\Kr(R,\widetilde{\star})_{M_2}\subseteq\Kr(R,\widetilde{\star})_{M_1}$
, $M_1$ and $M_2$ are maximal ideals, and $\Kr(R,\widetilde{\star})$
is a Pr\"{u}fer domain; hence $M_1=M_2$ (cf. \cite[Theorem
17.6(c)]{G}) and
$\Kr(R,\widetilde{\star})_{M_1}=(R_{N_{\star}(H)})_{P_1R_{N(H)}}$.
Thus
$$
\Kr(R,\widetilde{\star})=\bigcap_{M\in\Max(\Kr(R,\widetilde{\star}))}\Kr(R,\widetilde{\star})_M=\bigcap_{P\in
h\text{-}\QMax^{\widetilde{\star}}(R)}(R_{N_{\star}(H)})_{PR_{N_{\star}(H)}}=R_{N_{\star}(H)}.
$$

$(6)\Rightarrow(9)$ Assume that
$R_{N_{\star}(H)}=\Kr(R,\widetilde{\star})$. Let $I$ be a nonzero
homogeneous finitely generated ideal of $R$. Then by Lemma \ref{w}
and Theorem \ref{p}(3), we have
$I^{\widetilde{\star}}=IR_{N_{\star}(H)}\cap
R_H=I\Kr(R,\widetilde{\star})\cap R_H=I^{\widetilde{\star}_a}$.

$(9)\Rightarrow(1)$ Let $a$ and $b$ be two nonzero homogeneous
elements of $R$. Then
$((a,b)^3)^{\widetilde{\star}_a}=((a,b)(a^2,b^2))^{\widetilde{\star}_a}$
which implies that
$((a,b)^2)^{\widetilde{\star}_a}=(a^2,b^2)^{\widetilde{\star}_a}$.
Hence $((a,b)^2)^{\widetilde{\star}}=(a^2,b^2)^{\widetilde{\star}}$
and so $(a,b)^2R_{H\backslash P}=(a^2,b^2)R_{H\backslash P}$ for
each homogeneous maximal quasi-$\widetilde{\star}$-ideal $P$ of $R$.
On the other hand $R^{\widetilde{\star}}=R^{\widetilde{\star}_a}$ by
(9). Hence $R^{\widetilde{\star}}$ is integrally closed. Thus
$R^{\widetilde{\star}}R_{H\backslash P}=R_{H\backslash P}$ is
integrally closed. Therefore by Proposition \ref{grp},
$R_{H\backslash P}$ is a graded Pr\"{u}fer domain for each
homogeneous maximal quasi-$\star_f$-ideal of $R$. Thus $R$ is a
graded P$\star$MD by Theorem \ref{loc}.
\end{proof}

The following theorem is a graded version of a characterization of
Pr\"{u}fer domains proved by Davis \cite[Theorem 1]{D}. It also
generalizes \cite[Theorem 2.10]{DHLZ}, in the $t$-operation, and
\cite[Theorem 5.3]{EF}, in the case of semistar operations.

\begin{thm}\label{NKP2} Let $R=\bigoplus_{\alpha\in\Gamma}R_{\alpha}$ be a
graded integral domain with a unit of nonzero degree, and $\star$ be
a semistar operation on $R$ such that $R^{\star}\subsetneq R_H$.
Then, the following statements are equivalent:
\begin{itemize}
\item[(1)] $R$ is a graded P$\star$MD.
\item[(2)] Each homogeneously $(\star,t)$-linked overring of $R$ is a
P$v$MD.
\item[(3)] Each homogeneously $(\star,d)$-linked overring of $R$ is
a graded Pr\"{u}fer domain.
\item[(4)] Each homogeneously $(\star,t)$-linked overring of $R$, is integrally closed.
\item[(5)] Each homogeneously $(\star,d)$-linked overring of $R$, is integrally closed.
\end{itemize}
\end{thm}

\begin{proof} $(1)\Rightarrow(2)$ Let $T$ be a homogeneously $(\star,t)$-linked overring of
$R$. Thus by Lemma \ref{link}, we have $R_{N_{\star}(H)}\subseteq
T_{N_v(H)}$. Since $R$ is a graded P$\star$MD, by Theorem
\ref{NKP1}, we have $R_{N_{\star}(H)}$ is a Pr\"{u}fer domain. Thus
by \cite[Theorem 26.1]{G}, we have $T_{N_v(H)}$ is a Pr\"{u}fer
domain. Hence, again by Theorem \ref{NKP1}, we have $T$ is a graded
P$v$MD. Therefore using \cite[Theorem 6.4]{AA}, $T$ is a P$v$MD.

$(2)\Rightarrow(4)\Rightarrow(5)$ and $(3)\Rightarrow(5)$ are clear.

$(5)\Rightarrow(1)$ Let $P\in h$-$\QMax^{\widetilde{\star}}(R)$. For
a nonzero homogeneous $u\in R_H$, let $T=R[u^2,u^3]_{H\backslash
P}$. Then $R_{H\backslash P}$ and $T$ are homogeneous
$(\star,d)$-linked overring of $R$ by Example \ref{el}. So that
$R_{H\backslash P}$ and $T$ are integrally closed. Hence $u\in T$,
and since $T=R_{H\backslash P}[u^2,u^3]$, there exists a polynomial
$\gamma\in R_{H\backslash P}[X]$ such that $\gamma(u)=0$ and one of
the coefficients of $\gamma$ is a unit in $R_{H\backslash P}$. So
$u$ or $u^{-1}$ is in $R_{H\backslash P}$ by \cite[Theorem 67]{K}.
Therefore by Lemma \ref{lo}, $R_{H\backslash P}$ is a graded
Pr\"{u}fer domain. Thus $R$ is a graded P$\star$MD by Theorem
\ref{loc}.

$(1)\Rightarrow(3)$ Is the same argument as in part
$(1)\Rightarrow(2)$.
\end{proof}

The next result gives new characterizations of P$v$MDs for graded
integral domains, which is the special cases of Theorems \ref{loc},
\ref{hhh}, \ref{NKP1}, and \ref{NKP2}, for $\star=v$.

\begin{cor} Let $R=\bigoplus_{\alpha\in\Gamma}R_{\alpha}$ be a
graded integral domain with a unit of nonzero degree. Then, the
following statements are equivalent:
\begin{itemize}
\item[(1)] $R$ is a (graded) P$v$MD.
\item[(2)] $R_{H\backslash P}$ is a graded Pr\"{u}fer domain for
each $P\in h$-$\QMax^t(R)$.
\item[(3)] $R_P$ is a valuation domain for
each $P\in h$-$\QMax^t(R)$.
\item[(4)] Every ideal of
$R_{N_v(H)}$ is extended from a homogeneous ideal of $R$.
\item[(5)] $R_{N_v(H)}$ is a Pr\"{u}fer domain.
\item[(6)] $R_{N_v(H)}$ is a B\'{e}zout domain.
\item[(7)] $R_{N_v(H)}=\Kr(R,w)$.
\item[(8)] $\Kr(R,w)$ is a quotient ring of $R$.
\item[(9)] $\Kr(R,w)$ is a flat $R$-module.
\item[(10)] Each homogeneously $t$-linked overring of $R$ is a
P$v$MD.
\item[(11)] Each homogeneously $t$-linked overring of $R$, is integrally closed.
\item[(12)] $(C(f)C(g))^w=C(fg)^w$ for all
$f, g\in R_H$.
\item[(13)] $I^w=I^{w_a}$ for each nonzero homogeneous finitely
generated ideal of $R$.
\end{itemize}
\end{cor}

\begin{center} {\bf ACKNOWLEDGMENT}

\end{center}I am grateful to the referee for carefully reading the first version
of this article.

\end{document}